\documentclass{article}

\usepackage{amsmath,amsfonts,amssymb,amsthm,authblk,cases,hyperref,xcolor}
\usepackage[capitalise]{cleveref}
\usepackage{vmargin}
\setmarginsrb{2cm}{1cm}{2cm}{3cm}{1cm}{1cm}{2cm}{2cm}
\usepackage{dsfont}

\renewcommand{\div}{\operatorname{div}}

\newcommand{\Id}{\operatorname{Id}}

\newcommand{\rank}{\operatorname{rank}}
\renewcommand{\leq}{\leqslant}

\renewcommand{\geq}{\geqslant}

\renewcommand{\Re}{\operatorname{Re}}

\everymath{\displaystyle}
\numberwithin{equation}{section}

\newtheorem{Theorem}{Theorem}[section]

\newtheorem{Lemma}[Theorem]{Lemma}

\title{Feedback stabilization of parabolic systems with input delay}

\author[1]{Imene Aicha Djebour}
\author[1]{Tak\'eo Takahashi}
\author[1]{Julie Valein}
\affil[1]{Universit\'e de Lorraine, CNRS, Inria, IECL, F-54000 Nancy}

\date{\today}

\begin{document}

\maketitle

\begin{abstract}
This work is devoted to the stabilization of parabolic systems with a finite-dimensional control subjected to a constant delay. 
Our main result shows that the Fattorini-Hautus criterion yields the existence of such a feedback control, as in the case of stabilization without delay. 
The proof 
consists in splitting the system into a finite dimensional unstable part and a stable infinite-dimensional part and to apply the Artstein transformation on the finite-dimensional system to remove the delay in the control.
Using our abstract result, we can prove new results for the stabilization of parabolic systems with constant delay: 
the $N$-dimensional linear convection-diffusion equation with $N\geq 1$ and the Oseen system. We end the article by showing that this theory 
can be used to stabilize nonlinear parabolic systems with input delay by proving the 
local feedback distributed stabilization of the Navier-Stokes system around a stationary state.
\end{abstract} 
	
\vspace{1cm}
	
\noindent {\bf Keywords:} stabilizability, delay control, parabolic systems, finite-dimensional control
	
\noindent {\bf 2010 Mathematics Subject Classification}  93B52, 93D15, 35Q30, 76D05, 93C20.

\tableofcontents
	
\section{Introduction}
Time delay phenomena appear in many applications, for instance in biology, mechanics, automatic control or engineering 
and are inevitable due to the time-lag between the measurements and their exploitation. For instance in control problems, one need to take into account
the analysis time or the computation time. 
In the context of stability problems for partial differential equations with
delay, it is classical that a small delay in the feedback mechanism can destabilize a system
(see for instance \cite{MR818942, MR937679}). On the other hand, a delay term can also improve the performance of a
system (see for instance \cite{abdallah1993delayed}). These features appear for hyperbolic systems and here our aim is to consider the 
stabilization problems for a large class of parabolic systems with a particular delay input.

More precisely this article is devoted to the feedback stabilization of the system
\begin{equation}\label{ijt000}
z'=Az+Bv+f, \quad z(0)=z^0,
\end{equation}
where $A$ is the generator of an analytic semigroup $(e^{tA})_{t\geq 0}$ on a Hilbert space $\mathbb{H}$, 
where $B : \mathbb{U} \to \mathcal{D}(A^*)'$ is a linear operator on a Hilbert space $\mathbb{U}$
and where $f$ is a given source satisfying an exponential decay at infinity. 
Our aim is to obtain a feedback control $v(t)$ that depends on the values of $z(s)$ for $s\leq t-\tau$, where $\tau>0$ is a positive constant 
corresponding to a delay. With such a feedback control, our aim is to obtain exponential stabilization of \eqref{ijt000} if we assume that it is the case without delay.
A characterization of the exponential stabilization of \eqref{ijt000} in the case without delay is the well-known Fattorini-Hautus criterion
(see \cite{Fattorini1966}, \cite{Hautus} and \cite{BT2014}):
\begin{equation}
\forall\varepsilon\in \mathcal{D}(A^*), \ 
\forall \lambda\in\mathbb{C},\ 
\Re\lambda\geq -\sigma\quad
A^*\varepsilon= \lambda \varepsilon \quad \text{and} \quad B^*\varepsilon=0
\quad\Longrightarrow\quad 
\varepsilon = 0.  \tag{$\text{UC}_{\sigma}$}\label{UCstab}
\end{equation}
Here and in what follows, we denote by $A^* : \mathcal{D}(A^*) \to \mathbb{H}$ and by $B^* : \mathcal{D}(A^*) \to \mathbb{U}$ the adjoint operators of $A$ and $B$.
This criterion is equivalent to the exponential stabilization of \eqref{ijt000} with a rate larger than $\sigma$ provided we assume the following hypotheses:
\begin{equation}\label{H1}\tag{Hyp1}
\begin{array}{c}
\text{The spectrum of $A$ consists of isolated eigenvalues $(\lambda_j)$ with 
finite algebraic multiplicity $N_j$}\\
\text{and there is no finite cluster point in } \{\lambda \in \mathbb{C} \ ; \ \Re \lambda \geq -\sigma\},
\end{array}
\end{equation}
\begin{equation}\label{H2}\tag{Hyp2}
B\in \mathcal{L}(\mathbb{U},\mathbb{H}_{-\gamma}) \quad \text{for some } \gamma\in [0,1).
\end{equation}
The spaces $\mathbb{H}_\alpha$ are defined as follows: we fix $\mu_0\in \rho(A)$, then
\begin{equation} \label{Halpha}
\mathbb{H}_\alpha:= \left\{
\begin{array}{ll}
\mathcal{D}((\mu_0-A)^{\alpha}) &\mbox{ if }\alpha\geq 0\\
\mathcal{D}((\mu_0-A^*)^{-\alpha})' &\mbox{ if }\alpha< 0
\end{array}
\right.\quad \mbox{ and }\quad 
\mathbb{H}_{\alpha}^*:= \left\{
\begin{array}{ll}
\mathcal{D}((\mu_0-A^*)^{\alpha}) &\mbox{ if }\alpha\geq 0\\
\mathcal{D}((\mu_0-A)^{-\alpha})' &\mbox{ if }\alpha< 0.
\end{array}
\right.
\end{equation}
To deal with the source $f$, we also assume the following hypothesis
\begin{equation}\label{H3}\tag{Hyp3}
\mathbb{H}_\alpha=[\mathbb{H},\mathcal{D}(A)]_{\alpha} \quad (\alpha \in [0,1]),
\end{equation}
where $[\cdot,\cdot]_{\alpha}$ denotes the complex interpolation method.
We assume that
\begin{equation}\label{ijt004}
f_\sigma:t\mapsto e^{\sigma t} f(t)\in L^2(0,\infty;\mathbb{H}_{-\gamma'}) \quad \gamma'<1/2.
\end{equation}
We say that $f\in L_\sigma^2(0,\infty;\mathbb{H}_{-\gamma'})$ if  $f_\sigma\in L^2(0,\infty;\mathbb{H}_{-\gamma'})$ and we write
$$
\left\|f\right\|_{L_\sigma^2(0,\infty;\mathbb{H}_{-\gamma'})}= \left\|f_\sigma\right\|_{L^2(0,\infty;\mathbb{H}_{-\gamma'})}.
$$
The same definition can be extended to spaces of the kind $L^p_\sigma(0,\infty;\mathbb{X})$, $C_\sigma^0([0,\infty);\mathbb{X})$, $H^m_\sigma(0,\infty;\mathbb{X})$, with $\mathbb{X}$ a Banach space.

Note that a sufficient condition for \eqref{H1} is that $A$ has compact resolvent. 
For all $\lambda_j$ eigenvalue of $A$, we define its geometric multiplicity
$$
\ell_j:=\dim \ker (A-\lambda_j\Id)\in \mathbb{N}^*.
$$
We also define the maximum of the geometric multiplicities of the unstable modes:
\begin{equation}\label{Nc}
N_+:=\max \{\ell_i, \ \Re \lambda_i\geq -\sigma\}.
\end{equation}

Our main result is the following theorem:
\begin{Theorem}\label{thmain}
Let us consider $\sigma>0$ and let us 
assume \eqref{H1}, \eqref{H2}, \eqref{H3} and \eqref{UCstab}. 
Then there exist $K\in L^\infty_{\rm loc}(\mathbb{R}^2;\mathcal{L}(\mathbb{H}))$, $\zeta_k\in \mathcal{D}(A^*)$, $v_k\in B^*\left(\mathcal{D}(A^*)\right)$, $k=1,\ldots,N_+$, 
such that
if
\begin{equation}
\label{fb1}
v(t)=\mathds{1}_{[\tau,+\infty)}(t) \sum_{k=1}^{N_+} \left(z(t-\tau)+\int_0^{t-\tau} K(t-\tau,s)z(s) \ ds,\zeta_k\right)_{\mathbb{H}} v_k,
\end{equation}
then for any  $z^0 \in \mathbb{H}$, $f$ satisfying \eqref{ijt004}, the solution $z$ of \eqref{ijt000} satisfies 
\begin{equation}\label{ijt008}
\|z(t)\|_{\mathbb{H}} \leq C e^{-\sigma t} \left(\|z^0\|_{\mathbb{H}} + \|f\|_{L^2_{\sigma}(0,\infty;\mathbb{H}_{-\gamma'})}\right)
\quad (t>0).
\end{equation}
Assume moreover that $\gamma=0$, $\gamma'=0$ and that $z^0\in \mathbb{H}_{1/2}$. Then,
$$
z\in L^2_\sigma (0,\infty;\mathbb{H}_1)\cap C^0_\sigma ([0,\infty);\mathbb{H}_{1/2})\cap H^1_\sigma (0,\infty;\mathbb{H}),
$$
and 
\begin{equation}\label{ijt200}
\|z\|_{L^2_\sigma (0,\infty;\mathbb{H}_1)\cap C^0_\sigma ([0,\infty);\mathbb{H}_{1/2})\cap H^1_\sigma (0,\infty;\mathbb{H})}
\leq C\left(\|z^0\|_{\mathbb{H}_{1/2}} + \|f\|_{L^2_{\sigma}(0,\infty;\mathbb{H})}\right).
\end{equation}
\end{Theorem}
Here and in all what follows, $\mathds{1}_{\mathcal{O}}$ is the characteristic function of the set $\mathcal{O}$.

The above result shows that we can stabilize the above general class of linear parabolic systems with a finite number of controls and with a constant delay: the feedback control $v(t)$ 
at time $t$, given by \eqref{fb1}, only depends on value of the state $z(s)$ for $s\leq t-\tau$. 
This result can be seen as a generalization of several recent results on the stabilization of parabolic systems with delay control, in particular
\cite{MR3767485} where the authors constructed a feedback control for finite dimensional linear systems, and
\cite{MR3936420} where the authors obtained a stabilizing feedback control of a one-dimensional reaction-diffusion equation with a boundary control subjected to a constant delay. Let us mention some ideas of their method that we adapt to prove our result: using that their operator is 
self-adjoint of compact resolvent they split the system into an unstable finite-dimensional part and a stable infinite-dimensional part. They are thus led to stabilize the 
finite-dimensional unstable system and to do this with a delay, they use the Artstein transformation and obtain an autonomous control system without delay satisfying the Kalman condition. Finally, by using an appropriate Lyapunov function, they prove that the feedback control designed in the finite-dimensional part actually stabilizes the whole system.

We can mention several articles in this direction: in \cite{lhachemi2019boundary}, the authors consider the stabilization of a 
structurally damped Euler-Bernoulli beam. The corresponding system is parabolic but the main operator is no more self-adjoint.
Then \cite{lhachemi2019feedback} generalizes the result 
of \cite{MR3936420} in the case where the main operator is a Riesz spectral operator with simple eigenvalues. 
In \cite{Lhachemi_2020} they manage to extend the result of \cite{MR3936420} 
to the case where the control contains some disturbances and where the delay can depend on time. 


Here our aim is to extend the result of \cite{MR3936420} for a large class of parabolic systems, and in particular with the possibility to consider partial differential equations in dimension larger than one. We also precise the number of controls $N_+$ needed to stabilize the system by using the approach developed in  \cite{MB1} in the case of the Navier-Stokes system or in  \cite{BT2014}, for general linear and nonlinear parabolic systems. We present two important examples, that is the reaction-diffusion equation and the Oseen system and we end this paper to show that within this framework, we can also handle some nonlinear parabolic systems such as the Navier-Stokes system.

The present paper is organized as follows. In \cref{sec_proof}, the proof of \cref{thmain} is given. As in \cite{MR3936420}, it relies on the decomposition of the system \eqref{ijt000} into two parts: an unstable finite-dimensional part and an infinite-dimensional part. 
This decomposition is possible thanks to \eqref{H1} and \cite[Theorem 6.17, p.178]{Kato1966}. 
Due to the presence of a constant delay, an equivalent autonomous control system is considered for the finite-dimensional part by means of the Artstein transformation.
This system is exponentially stabilizable by using \eqref{UCstab}. Using the inverse of the Artstein transform, a stabilizing feedback control is designed in the finite-dimensional space that stabilizes exponentially the finite-dimensional unstable system (with delay control).
Finally, we prove that the designed feedback stabilizes exponentially the complete system. 
Thereafter, we illustrate our results by some precise examples: 
the case of the feedback stabilization of the $N$-dimensional linear convection-diffusion equation with $N\geq 1$ with delay boundary  control in \cref{sec_heat}, 
the case of the feedback stabilization of the Oseen system with delay distributed control in \cref{sec_oseen} 
and finally, a local feedback distributed stabilization of the Navier-Stokes system around a stationary state in \cref{sec_navier}.

\section{Proof of \cref{thmain}}\label{sec_proof}
Let us consider $\sigma>0$.
We first decompose the spectrum of $A$ into the ``unstable'' modes and the ``stable'' modes:
\begin{equation}\label{ijt005}
\Sigma_+:= \{\lambda_j \ ; \ \Re \lambda_j \geq -\sigma\},
\quad
\Sigma_-:= \{\lambda_j \ ; \ \Re \lambda_j < -\sigma\}.
\end{equation}
Using that $(e^{tA})_{t\geq 0}$ is an analytic semigroup (see \cite[Theorem 2.11, p.112]{BDDM}) and \eqref{H1}, we see that $\Sigma_+$ is of finite cardinal.

Thus, we can introduce the projection operator (see \cite[Thm. 6.17, p.178]{Kato1966}) defined by 
\begin{equation}
P_+:= \frac{1}{2\pi \imath}\int_{\Gamma_+} (\lambda -A)^{-1} \ d\lambda,
\end{equation}
where $\Gamma_+$ is a contour enclosing $\Sigma_+$ but no other point of the spectrum of $A$. 
We can define 
$$
\mathbb{H}_+:=P_+ \mathbb{H}, \quad \mathbb{H}_-:=(\Id-P_+) \mathbb{H}.
$$
From \cite[Thm. 6.17, p.178]{Kato1966}, we have $\mathbb{H}_+ \oplus \mathbb{H}_-=\mathbb{H}$ and if we set
$$
A_+:=A_{|\mathbb{H}_+} : \mathbb{H}_+ \to \mathbb{H}_+, \quad 
A_-:=A_{|\mathbb{H}_-} : \mathcal{D}(A)\cap \mathbb{H}_- \to \mathbb{H}_-, 
$$
then the spectrum of $A_+$ (resp. $A_-$) is exactly $\Sigma_+$ (resp. $\Sigma_-$). 
By using the analyticity of $\left(e^{A t}\right)_{t\geq 0}$, \eqref{H1} and \eqref{ijt005}, we deduce the existence of $\sigma_- >\sigma$ such that
\begin{equation}\label{ijt006}
\left\| e^{A_-t} \right\|_{\mathcal{L}(\mathbb{H}_-)} \leq Ce^{-\sigma_- t},
\quad
\left\| (\lambda_0-A)^{\gamma} e^{A_-t} \right\|_{\mathcal{L}(\mathbb{H}_-)} \leq C\frac{1}{t^\gamma} e^{-\sigma_- t}.
\end{equation}

We can proceed similarly for $A^*$: we write
\begin{equation}
P_+^*:= \frac{1}{2\pi \imath}\int_{\overline{\Gamma_+}} (\lambda -A^*)^{-1} \ d\lambda,
\end{equation}
$$
\mathbb{H}_+^*:=P_+^* \mathbb{H}, \quad \mathbb{H}_-^*:=(\Id-P_+^*) \mathbb{H},
$$
$$
A_+^*:=A_{|\mathbb{H}_+^*} : \mathbb{H}_+^* \to \mathbb{H}_+^*, \quad 
A_-^*:=A_{|\mathbb{H}_-^*} : \mathcal{D}(A^*)\cap \mathbb{H}_-^* \to \mathbb{H}_-^*. 
$$
Note that $P_+^*$ is the adjoint of $P_+$. In particular, we see that if $z\in \mathbb{H}_-$ and $\zeta\in  \mathbb{H}_+^*$,
then
\begin{equation}\label{oubli1}
\left(z,\zeta\right)_{\mathbb{H}}=\left((\Id-P_+)z,\zeta\right)_{\mathbb{H}}=\left(z,(\Id-P_+^*)\zeta\right)_{\mathbb{H}}=0.
\end{equation}

We also define
$$
\mathbb{U}_+:=B^* \mathbb{H}_+^*, \quad
\mathbb{U}_-:=B^* \mathbb{H}_-^*,
$$
and
$$
p_+ : \mathbb{U} \to \mathbb{U}_+, \quad
p_- : \mathbb{U} \to \mathbb{U}_-, \quad
i_+ : \mathbb{U}_+ \to \mathbb{U}, \quad
i_- : \mathbb{U}_- \to \mathbb{U},
$$
the orthogonal projections and the inclusion maps. Then we write
$$
B_+:=P_+ B i_+ \in \mathcal{L}(\mathbb{U}_+,\mathbb{H}_+), 
\quad
B_-:=(\Id-P_+) B i_- \in \mathcal{L}(\mathbb{U}_-,\mathbb{H}_-\cap \mathbb{H}_{-\gamma}). 
$$
We have that
$$
P_+B=B_+p_+
\quad
(\Id-P_+)B=B_-p_-.
$$
Indeed, for any $w\in \mathcal{D}(A)$ and $ v\in \mathbb{U}$  we can write
$$
\left\langle P_+Bv,w\right\rangle_{\mathbb{H}_{-1},\mathbb{H}_{1}}
=\left\langle v, B^*P^*_+ w \right\rangle_\mathbb{U} 
=\left\langle p_+v, p_+B^*P^*_+ w \right\rangle_\mathbb{U}
=\left\langle P_+B i_+p_+v,  w \right\rangle_{\mathbb{H}_{-1},\mathbb{H}_{1}}
=\left\langle B_+p_+v,  w \right\rangle_{\mathbb{H}_{-1},\mathbb{H}_{1}},
$$
and we can prove similarly the other relation.

From the above notation, we can split \eqref{ijt000} into the two equations (see \cite{BT2014,RaymondThevenet2009}). 
\begin{equation}\label{ijt001}
z_+'=A_+ z_+ + B_+ p_+ v + P_+f, \quad z_+(0)=P_+z^0,
\end{equation}
\begin{equation}\label{ijt002}
z_-'=A_- z_- + B_- p_- v + (\Id-P_+)f, \quad z_-(0)=(\Id-P_+)z^0.
\end{equation}
In order to study the stabilization of the finite-dimensional system \eqref{ijt001}, we use the 
Artstein transformation that allows us to pass from \eqref{ijt001} in the case of a delay input to an autonomous system.
More precisely, we consider 
$$
w(t):=z_+(t)+\int_t^{t+\tau} e^{(t-s)A_+} B_+ p_+ v(s) \ ds.
$$
Then in what follows, we study the stabilization of the autonomous system satisfied by $w$ (\cref{L1}). Since the corresponding feedback is expressed with $w$, we also
consider the inverse of the Artstein transformation and more precisely show the existence of a kernel $K$ to write $w$ in terms of $z_+$ (\cref{L2}).
\begin{Lemma}\label{L1}
Assume \eqref{UCstab} for $\sigma>0$. 
Then, there exist $C>0$ and $G\in\mathcal{L}(\mathbb{H}_+,\mathbb{U}_+)$, with $\rank G\leq N_+$ 
where $N_+$ is defined by \eqref{Nc},
such that 
for any $f\in L_\sigma^2(0,\infty;\mathbb{H}_{-\gamma'})$ and $w^0 \in \mathbb{H}_+$, 
the solution of
\begin{equation}
\label{eq12}
\left\{
\begin{array}{c}
w'=A_+w+e^{-\tau A_+}B_+p_+ Gw+P_+f,\\
w(0)=w^0 ,
\end{array}
\right.
\end{equation}
satisfies
\begin{equation}
\label{stab1}
\|w\|_{H^1_\sigma(0,\infty;\mathbb{H}_+)} \leq C\left(\|w^0\|_{\mathbb{H}_+} + \|P_+ f\|_{L^2_{\sigma}(0,\infty;\mathbb{H}_+)}\right).
\end{equation}
\end{Lemma} 
\begin{proof}
Let us consider $\sigma_\star>\sigma.$
First we notice that $(A_+,e^{-\tau A_+}B_+p_+)$ satisfies the Fattorini-Hautus test:
assume 
$$
A_+^* \varepsilon =\overline{\lambda_j} \varepsilon, \quad
B_+^* e^{-\tau A_+^*} \varepsilon=0.
$$
Then we deduce
$$
A^*\varepsilon =\overline{\lambda_j} \varepsilon, \quad
B_+^* e^{-\tau A_+^*} \varepsilon
=e^{-\tau \overline{\lambda_j} } B^* P_+^*  \varepsilon
=e^{-\tau \overline{\lambda_j} } B^* \varepsilon=0.
$$
Thus from \eqref{UCstab}, we deduce $\varepsilon=0$. We can thus use the standard result of Fattorini or Hautus
(see also \cite[Theorem 1.6]{BT2014}) to deduce that there exists $G\in\mathcal{L}(\mathbb{H}_+,\mathbb{U}_+)$, with $\rank G\leq N_+$ 
such that the solution of 
\begin{equation}
\label{eq12-2}
\left\{
\begin{array}{c}
w'=A_+w+e^{-\tau A_+}B_+ p_+ Gw,\\
w(0)=w^0 \in \mathbb{H}_+,
\end{array}
\right.
\end{equation}
satisfies
\begin{equation}
\label{stab1-2}
\left\|w(t)\right\|_{\mathbb{H}_+} \leq Ce^{-\sigma_\star t}\left\|w^0 \right\|_{\mathbb{H}_+} , \quad t\geq 0.
\end{equation}
Then we write the Duhamel formula for the solutions of \eqref{eq12} 
$$
w(t)=e^{A_\star t} w^0 +\int_0^t e^{A_\star(t-s)} P_+f(s) \ ds,
$$
with 
$$
A_\star=A_+ + e^{-\tau A_+}B_+ p_+ G,
$$
to deduce \eqref{stab1}.
\end{proof}

\begin{Lemma}\label{L2}
There exists $K\in L^\infty_{\rm loc}(\mathbb{R}^2;\mathcal{L}(\mathbb{H}_+))$ such that
\begin{multline}
\label{k3}
K(t,s)= e^{(t-s-\tau)A_+}B_+p_+G\mathds{1}_{(\max\{t-\tau,0\},t)} (s) \\+ \int_{\max\{t-\tau,s\}}^te^{(t-\xi-\tau)A_+}B_+p_+GK(\xi,s) \ d\xi 
\quad (t>0, s\in (0,t)).
\end{multline}
\end{Lemma}
\begin{proof}
The proof relies on a fixed point argument.
We set 
$$
K_0(t):=e^{(t-\tau)A_+}B_+p_+G, \quad K_0 \in L^\infty(0,\tau;\mathcal{L}(\mathbb{H_+})),
$$
so that \eqref{k3} writes
$$
K(t,s)=K_0(t-s)\mathds{1}_{(\max\{t-\tau,0\},t)} (s) 
+ \int_{\max\{t-\tau,s\}}^t K_0(t-\xi) K(\xi,s) \ d\xi.
$$

Let $T>0$, and let us define
$$
D_T=\{(t,s)\in\mathbb{R}^2,\quad t\in (0,T),\quad s\in (0,t)\}	,
$$
and
$$
\Phi:  L^\infty(D_T;\mathcal{L}(\mathbb{H_+}))\rightarrow L^\infty(D_T;\mathcal{L}(\mathbb{H_+})),
$$
$$
(\Phi K)(t,s)=\int_{\max\{t-\tau,s\}}^tK_0(t-\xi)K(\xi,s) \ d\xi \quad ((t,s)\in D_T).
$$
The mapping $\Phi$ is well-defined, and is a linear and bounded operator of $L^\infty(D_T;\mathcal{L}(\mathbb{H_+}))$. Moreover,
$$
\left\| (\Phi K)(t,s) \right\|_{\mathcal{L}(\mathbb{H}_+)} \leq t \left\|K_0 \right\|_{L^\infty(0,\tau;\mathcal{L}(\mathbb{H_+}))}
\left\|K \right\|_{L^\infty(D_T;\mathcal{L}(\mathbb{H_+}))}.
$$
This yields 
$$
\left\| (\Phi^2 K)(t,s) \right\|_{\mathcal{L}(\mathbb{H}_+)} \leq \frac{t^2}{2} \left\|K_0 \right\|_{L^\infty(0,\tau;\mathcal{L}(\mathbb{H_+}))}^2
\left\|K \right\|_{L^\infty(D_T;\mathcal{L}(\mathbb{H_+}))},
$$
and by induction
$$
\left\| (\Phi^n K)(t,s) \right\|_{\mathcal{L}(\mathbb{H}_+)} \leq \frac{t^n}{n!} \left\|K_0 \right\|_{L^\infty(0,\tau;\mathcal{L}(\mathbb{H_+}))}^n
\left\|K \right\|_{L^\infty(D_T;\mathcal{L}(\mathbb{H_+}))} \quad (n\in \mathbb{N}^*).
$$
In particular, for $n$ large enough, $\Phi^n$ is a strict contraction and so is
$$
(\widetilde{\Phi} K)(t,s):= (\Phi K)(t,s) + K_0(t-s)\mathds{1}_{(\max\{t-\tau,0\},t)} (s).
$$
Thus, $\widetilde \Phi$ admits a unique fixed point which is a solution of \eqref{k3}.

\end{proof}

We are now in a position to prove the main result
\begin{proof}[Proof of \cref{thmain}]
We consider $G$ and $K(t,s)$ obtained in \cref{L1} and \cref{L2}, and we set
\begin{equation}
\label{fb12}
v(t)=\mathds{1}_{[\tau,+\infty)}(t)G\left[z_+(t-\tau)+\int_0^{t-\tau} K(t-\tau,s)z_+(s) \ ds \right].
\end{equation}
Note that, we can write $G$ as 
$$
G(\phi)=\sum_{k=1}^{N_+} \left(\phi,\zeta_k\right)_{\mathbb{H}} v_k, \quad (\phi\in \mathbb{H}_+)
$$
with $\zeta_k\in \mathbb{H}_+^*$ and $v_k\in \mathbb{U}_+$, $k=1,\ldots,N_+$.
We can take $\zeta_k\in \mathbb{H}_+^*$ due to a standard result in linear algebra: combining $\dim \mathbb{H}_+^*=\dim \mathcal{L}(\mathbb{H}_+,\mathbb{C})$
and \eqref{oubli1}, we can identify these two spaces. 
The interest to take $\zeta_k\in \mathbb{H}_+^*$ is that the above formula for $G$ can be applied to $\phi\in \mathbb{H}$ and extend $G$ as a linear bounded operator in $\mathbb{H}$ satisfying $G=0$ in $\mathbb{H}_-$ (see \eqref{oubli1}). Extending also the family $K$ by $K(t,s)=0$ in $\mathbb{H}_-$, we see that \eqref{fb12} can be written
as \eqref{fb1}.

Let us define 
\begin{equation}\label{eqW}
w(t):=z_+(t)+\int_0^{t} K(t,s)z_+(s) \ ds,
\end{equation}
so that \eqref{ijt001} can be written
\begin{equation}\label{eqZ+2}
\left\{
\begin{array}{cc}
z_+'(t)=A_+ z_+(t)+B_+p_+\mathds{1}_{[\tau,+\infty)}(t)G w(t-\tau)+P_+f(t) \quad t>0,\\
z_+(0)=P_+z^0.
\end{array}
\right.
\end{equation}

Then we use \eqref{k3}, \eqref{eqW} and the Fubini theorem to perform the following computation for $t>0$:
\begin{multline}
\label{calculs}
\int_{t}^{t+\tau}e^{(t-s)A_+}B_+p_+ Gw(s-\tau) \mathds{1}_{[\tau,+\infty)}(s)   \ ds
=
\int_{\max \{t-\tau,0\}}^{t}e^{(t-s-\tau)A_+}B_+p_+ Gw(s) \ ds
\\
=\int_{\max \{t-\tau,0\}}^{t}e^{(t-s-\tau)A_+}B_+p_+ G\left[ z_+(s)+\int_0^{s} K(s,\xi)z_+(\xi) \ d\xi\right] \ ds
\\
=\int_{0}^{t}
\left[\mathds{1}_{(\max\{t-\tau,0\},t)} (s) e^{(t-s-\tau)A_+}B_+p_+ G
+\int_{\max \{t-\tau,s\}}^{t} e^{(t-\xi-\tau)A_+}B_+p_+ G  K(\xi,s) \ d\xi
\right]z_+(s)  \ ds
\\
=\int_{0}^{t} K(t,s)z_+(s)  \ ds
=w(t)-z_+(t).
\end{multline} 
%

Consequently,
\begin{equation}\label{tempo}
w(t)=z_+(t)+\int_{t}^{t+\tau}e^{(t-s)A_+}B_+p_+ Gw(s-\tau) \mathds{1}_{[\tau,+\infty)}(s)   \ ds.
\end{equation}
%
%
From \eqref{eqZ+2}, we deduce that $w$ is solution of \eqref{eq12} with $w^0=z_+(0)$.
Thus $w$ satisfies \eqref{stab1} and from \eqref{tempo}, 
\begin{equation}
\label{etoile}
\|z_+\|_{H^1_{\sigma}(0,\infty;\mathbb{H}_+)} \leq C\left(\|P_+ z^0\|_{\mathbb{H}_+} + \|P_+f\|_{L^2_{\sigma}(0,\infty;\mathbb{H}_+)}\right).
\end{equation}
This yields in particular that if $z^0\in \mathbb{H}_{1/2}$ and if $f\in L^2_{\sigma}(0,\infty;\mathbb{H})$, then
\begin{equation}
\label{stab1-strong}
\|z_+\|_{L^2_\sigma (0,\infty;\mathbb{H}_1)\cap C^0_\sigma ([0,\infty);\mathbb{H}_{1/2})\cap H^1_\sigma (0,\infty;\mathbb{H})}
\leq C\left(\|z^0\|_{\mathbb{H}_{1/2}} + \|f\|_{L^2_{\sigma}(0,\infty;\mathbb{H})}\right).
\end{equation}

Then, we can consider the solution of \eqref{ijt002}: for $t\geq \tau$,
\begin{multline}
z_-(t)=e^{A_-t} (\Id-P_+)z^0
+\int_\tau^t (\lambda_0-A)^{\gamma} e^{A_-(t-s)} (\lambda_0-A)^{-\gamma} B_- p_- Gw(s-\tau) \ ds 
\\
+\int_0^t e^{A_-(t-s)} (\Id-P_+)f(s) \ ds.
\end{multline}
Using \eqref{ijt006} and \eqref{stab1}, we deduce that 
\begin{multline*}
\|z_-(t)\|_{\mathbb{H}} \leq Ce^{-\sigma_- t} \left\|z^0\right\|_{\mathbb{H}}
+C e^{-\sigma t} \int_\tau^t \frac{1}{(t-s)^\gamma} e^{-(\sigma_- -\sigma) (t-s)}\ ds \left(\|P_+ z^0\|_{\mathbb{H}_+} + \|P_+ f\|_{L^2_{\sigma}(0,\infty;\mathbb{H}_+)}\right)  
\\
+C e^{-\sigma t} \int_0^t  \frac{1}{(t-s)^{\gamma'}} e^{-(\sigma_- -\sigma) (t-s)} \left\| e^{\sigma s} f(s)\right\|_{\mathbb{H}_{-\gamma'}} \ ds.
\end{multline*}
Using that $\sigma_->\sigma$, $\gamma<1$ and $\gamma'<1/2$, we deduce from the above estimate that
$$
\|z_-(t)\|_{\mathbb{H}_-} \leq C e^{-\sigma t} \left(\|z^0\|_{\mathbb{H}} + \|f\|_{L^2_{\sigma}(0,\infty;\mathbb{H}_{-\gamma'})}\right)
\quad (t>0).
$$
Combing this with \eqref{etoile}, we deduce \eqref{ijt008}.

Let us prove now \eqref{ijt200}. 
If $f\in L^2_{\sigma}(0,\infty;\mathbb{H})$, $B\in \mathcal{L}(\mathbb{U},\mathbb{H})$ and if $z^0\in \mathbb{H}_{1/2}$, then the first part remains unchanged,
and we have \eqref{stab1-strong} and
\begin{equation}
\label{ijt100}
\|v\|_{L^2_{\sigma}(0,\infty;\mathbb{U})} \leq C\left(\|z^0\|_{\mathbb{H}_+} + \|f\|_{L^2_{\sigma}(0,\infty;\mathbb{H}_+)}\right).
\end{equation}
Consequently,
\begin{equation}\label{ijt101}
B_- p_- v + (\Id-P_+)f \in L^2_{\sigma}(0,\infty;\mathbb{H}_-),
\end{equation}
and
$$
z_-(0)=z^0-P_+z^0 \in \mathbb{H}_{1/2}\cap \mathbb{H}_-.
$$
Using that $A_-$ is the infinitesimal generator of an analytic semigroup of type smaller than $-\sigma$
(see, for instance \cite[Proposition 2.9, p.120]{BDDM}), then
$$
z_-\in L^2_\sigma (0,\infty;\mathbb{H}_1)\cap C^0_\sigma ([0,\infty);\mathbb{H}_{1/2})\cap H^1_\sigma (0,\infty;\mathbb{H}),
$$
and from \eqref{ijt100}
\begin{multline*}
\|z_-\|_{L^2_\sigma (0,\infty;\mathbb{H}_1)\cap C^0_\sigma ([0,\infty);\mathbb{H}_{1/2})\cap H^1_\sigma (0,\infty;\mathbb{H})}
\leq C\left(\|z^0-P_+z^0 \|_{\mathbb{H}_{1/2}} + \| B_- p_- v\|_{L^2_{\sigma}(0,\infty;\mathbb{H})} + \|(\Id-P_+)f\|_{L^2_{\sigma}(0,\infty;\mathbb{H})}\right)\\
\leq C\left(\|z^0\|_{\mathbb{H}_{1/2}} + \|f\|_{L^2_{\sigma}(0,\infty;\mathbb{H})}\right).
\end{multline*}
Combining this with \eqref{stab1-strong}, we deduce \eqref{ijt200}.
\end{proof}

\section{Feedback boundary stabilization of the convection-diffusion equation}\label{sec_heat}
Let $\Omega\subset \mathbb{R}^N$ ($N\geq 1$) be a bounded domain of class $C^{1,1}$. 
In this section, we apply \cref{thmain} for the stabilization of the convection-diffusion equation. Let us consider $\Gamma$ a non-empty open subset of $\partial \Omega$
and the control problem:
\begin{equation}\label{sys}
\left\{
\begin{array}{rl}
	\partial_t z-\Delta z-b\cdot \nabla z-cz=0  & \text{in} \ (0,\infty)\times \Omega,\\
	z=v& \text{on} \ (0,\infty)\times \Gamma,\\
	z=0& \text{on} \ (0,\infty)\times (\partial \Omega\setminus\Gamma),\\
	z(0,\cdot)=z^0 &  \text{in}\ \Omega,
\end{array}
\right.
\end{equation}
where $c,b,\div b\in L^\infty(\Omega)$.
In order to write \eqref{sys} under the form \eqref{ijt000}, we introduce the following functional setting:
$$
\mathbb{H}=L^2(\Omega),\quad \mathbb{U}=L^{2}(\Gamma),
$$
$$
Az=\Delta z+b\cdot \nabla z+cz,\quad \mathcal{D}(A)=H^2(\Omega)\cap H^1_0(\Omega).
$$ 
From standard results on this operator $A$ (see for example \cite[Theorem 5, p.305]{MR1625845}), we see that \eqref{H1} holds true.
To define the control operator $B$, we use a standard method (see, for instance \cite[pp.341-343]{TucsnakWeiss} or \cite{RaymondStokes}):
we first consider the lifting operator $D_0 \in\mathcal{L}(L^{2}(\partial\Omega); L^2(\Omega))$
such that for any $v\in L^2(\partial\Omega)$, $w=D_0 v$ is the unique solution of the following system
\begin{equation*}
\left\{
\begin{array}{rl}
\lambda_0 w-\Delta w-b\cdot \nabla w-cw=0& \text{ in } \Omega,\\
 w=v &  \text{ on } \ \partial\Omega,
\end{array}
\right.
\end{equation*}
where $\lambda_0\in \rho(A)$. 
Then, we set
$$
B = (\lambda_0-A)D_0 : \mathbb{U}\longrightarrow (\mathcal{D}(A^*))',
$$
where we have extended the operator $A$ as an operator from $L^2(\Omega)$ into $(\mathcal{D}(A^*))'$
and where we see $\mathbb{U}$ as a closed subspace of $L^2(\partial \Omega)$ (by extending by zero in $\partial \Omega\setminus\Gamma$ any $v\in \mathbb{U}$).
Using standard results on elliptic equations, we have that $B$ satisfies \eqref{H2} for any $\gamma>3/4$.

Let us recall how we can see that with $A$ and $B$ defined as above
\eqref{sys} writes as \eqref{ijt000}. We set $\widetilde z=z-w$, with $w=D_0v$. Then
$\widetilde z$ satisfies the system
\begin{equation*}
\left\{\begin{array}{rl}
\partial_t \widetilde z-\Delta \widetilde z -b\cdot \nabla \widetilde z-c\widetilde z=-\partial_t w+\lambda_0 w 
	& \text{in} \ (0,\infty)\times \Omega,\\
\widetilde z=0 &  \text{on} \ (0,\infty)\times \partial\Omega,\\
\widetilde z(0,\cdot)=\widetilde z^0:=z^0-w(0,\cdot)&  \text{ in } \ \Omega.
\end{array}
\right.
\end{equation*}
Using the Duhamel formula, we have
$$
\widetilde z(t)=e^{tA}\widetilde z^0+\int_0^t e^{(t-s)A}(-\partial_t w(s)+\lambda_0 w(s)) \ ds.
$$
By integrating by parts, we obtain
$$
z(t)=e^{tA}z^0+\int_0^t e^{(t-s)A}(\lambda_0-A) w(s)\ ds,
$$
that is
\begin{equation*}
\left\{
\begin{array}{rl}
z'=Az+(\lambda_0-A) D_0 v,\\
z(0)=z^0.
\end{array}
\right.
\end{equation*}

To apply \cref{thmain}, we only need to check \eqref{UCstab}. We recall that
$$
\mathcal{D}(A^*)=H^2(\Omega)\cap H^1_0(\Omega),
\quad
A^*\varepsilon=\Delta \varepsilon-\overline b\cdot \nabla \varepsilon+(\overline{c-\div b})\varepsilon,
$$
(see, for instance, \cite[p.345]{TucsnakWeiss}).
Moreover, by classical results (see \cite[Proposition 10.6.7]{TucsnakWeiss}), we see that
$$
D_0^* :=-\frac{\partial}{\partial \nu} (\lambda_0-A^*)^{-1},
$$
and thus
$$
B^*\varepsilon :=-\frac{\partial \varepsilon}{\partial \nu}_{|\Gamma}.
$$
Thus if $\varepsilon$ satisfies $A^*\varepsilon= \lambda \varepsilon$ and $B^*\varepsilon=0$, then
\begin{equation*}
\left\{
\begin{array}{rl}
\lambda \varepsilon-  \Delta \varepsilon+\overline b\cdot \nabla \varepsilon-(\overline{c-\div b})\varepsilon=0& \text{in}\ \Omega,\\
\varepsilon=0 &  \text{on} \ \partial\Omega,\\
\frac{\partial \varepsilon}{\partial \nu}=0&  \text{on} \ \Gamma.
\end{array}
\right.
\end{equation*}
From standard results on the unique continuation of the Laplace operator (see for instance \cite[Theorem 5.3.1, p.125]{MR0404822}), we deduce that $\varepsilon=0$. Thus \eqref{UCstab} holds for any 
$\sigma$ and we deduce the following result by applying \cref{thmain}:
\begin{Theorem}\label{stabheat}
Assume $\sigma>0$ and let us define $N_+$ by \eqref{Nc}.
Then there exist $K\in L^\infty_{\rm loc}(\mathbb{R}^2;\mathcal{L}(L^2(\Omega)))$, 
$\zeta_k\in H^2(\Omega)\cap H^1_0(\Omega)$, $v_k\in H^{1/2}(\Gamma)$, $k=1,\ldots,N_+$, such that
the solution $z$ of \eqref{sys} with
\begin{equation}
\label{fb1-heat}
v(t)=\mathds{1}_{[\tau,+\infty)}(t) \sum_{k=1}^{N_+} \left(\int_{\Omega} \left[z(t-\tau)+\int_0^{t-\tau} K(t-\tau,s)z(s) \ ds\right]\zeta_k \ dx\right) v_k,
\end{equation}
and for $z^0 \in L^2(\Omega)$ satisfies 
\begin{equation}\label{ijt008-heat}
\|z(t)\|_{L^2(\Omega)} \leq C e^{-\sigma t} \|z^0\|_{L^2(\Omega)} .
\end{equation}
\end{Theorem}

\section{Feedback distributed stabilization of the Oseen system}\label{sec_oseen}
Let $\Omega\subset \mathbb{R}^3$ be a bounded domain of class $C^{1,1}$. 
In this section, we apply \cref{thmain} to the Oseen system:
\begin{equation}
\label{oseen}
\left\{
\begin{array}{rl}
\partial_t z+(w^S\cdot \nabla)z+(z\cdot \nabla)w^S-\nu \Delta z+\nabla q=\mathds{1}_{\mathcal{O}} v  & \text{in}\ (0,\infty)\times \Omega,\\
\nabla \cdot z=0& \text{in} \ (0,\infty)\times \Omega,\\
z=0 & \text{on} \ (0,\infty)\times \partial\Omega,\\
z(0,\cdot)=z^0 & \text{in} \ \Omega,
\end{array}
\right.
\end{equation}
where $w^S\in [H^2(\Omega)]^3$ is a fixed (real) velocity and $v$ is the control that acts on the nonempty open subset $\mathcal{O}\subset \Omega$. 
We could also consider the boundary stabilization of the Oseen system by using the same method as in the above section but with some adaptations due the incompressibility condition and due to the pressure (see \cite{MB1} for more details).

Let us give the functional setting:
$$
\mathbb{H}=\{z\in [L^2(\Omega)]^3,\quad \nabla\cdot z=0 \text{ in }\Omega,\quad z\cdot n=0\text{ on }\partial\Omega\},
\quad
\mathbb{U}=[L^2(\mathcal{O})]^3.
$$
We denote by $\mathbb{P}$ the orthogonal projection $\mathbb{P} : [L^2(\Omega)]^3\to \mathbb{H}$ and we define the Oseen operator:
$$
\mathcal{D}(A)= [H^2(\Omega)\cap H^1_0(\Omega)]^3 \cap \mathbb{H},
\quad
Az=\mathbb{P}\left(\nu\Delta z-(w^S\cdot \nabla)z-(z\cdot \nabla)w^S\right).
$$
We recall (see, for instance \cite[Theorem 20]{MB1}) that the operator $A$ is the infinitesimal generator of 
an analytic semigroup on $\mathbb{H}$ and has a compact resolvent. Moreover,
$$
\mathcal{D}(A^*)= [H^2(\Omega)\cap H^1_0(\Omega)]^3 \cap \mathbb{H},
\quad
A^* \varepsilon=\mathbb{P}\left(\nu\Delta \varepsilon+(w^S\cdot \nabla)\varepsilon-(\nabla w^S)^* \varepsilon\right).
$$

We also define the control operator $B\in \mathcal{L}(\mathbb{U},\mathbb{H})$ by
$$
Bv= \mathbb{P} \left(\mathds{1}_{\mathcal{O}} v\right),
$$
and we can check that
$$
B^* \varepsilon = \varepsilon_{|\mathcal{O}}.
$$
In particular, we see that \eqref{H1} and \eqref{H2} hold true and 
if $\varepsilon$ satisfies $A^*\varepsilon= \lambda \varepsilon$ and $B^*\varepsilon=0$, then
\begin{equation*}
\left\{
\begin{array}{rl}
\lambda \varepsilon-  \nu \Delta \varepsilon-(w^S\cdot \nabla)\varepsilon+(\nabla w^S)^* \varepsilon+\nabla \pi=0& \text{in}\ \Omega,\\
\nabla \cdot \varepsilon =0& \text{in}\ \Omega,\\
\varepsilon=0 & \text{on} \ \partial\Omega,\\
\varepsilon\equiv 0& \text{in} \ \mathcal{O}.
\end{array}
\right.
\end{equation*}
Then using \cite{FabreLebeau}, we deduce that $\varepsilon=0$. Thus \eqref{UCstab} holds for any 
$\sigma$ and we deduce the following result by applying \cref{thmain}:
\begin{Theorem}\label{staboseen}
Assume $\sigma>0$ and let us define $N_+$ by \eqref{Nc}.
Then there exist $K\in L^\infty_{\rm loc}(\mathbb{R}^2;\mathcal{L}(\mathbb{H}))$, 
$\zeta_k\in \mathcal{D}(A^*)$, $v_k\in [L^2(\mathcal{O})]^3$, $k=1,\ldots,N_+$, such that
the solution $z$ of \eqref{oseen} with
\begin{equation}
\label{fb1-oseen}
v(t)=\mathds{1}_{[\tau,+\infty)}(t) \sum_{k=1}^{N_+} \left(\int_{\Omega} \left[z(t-\tau)+\int_0^{t-\tau} K(t-\tau,s)z(s) \ ds\right]\zeta_k \ dx\right) v_k,
\end{equation}
and for $z^0 \in \mathbb{H}$ satisfies 
\begin{equation}
\|z(t)\|_{[L^2(\Omega)]^3} \leq C e^{-\sigma t} \|z^0\|_{[L^2(\Omega)]^3} .
\end{equation}
\end{Theorem}

Let us define 
\begin{equation}\label{ijt105}
\mathbb{V}= [H^1_0(\Omega)]^3 \cap \mathbb{H},
\end{equation}
then we have that $\mathbb{V}=\mathbb{H}_{1/2}$ (see again \cite[Theorem 20]{MB1}). Thus applying \cref{thmain},
we have also the following result on 
\begin{equation}
\label{oseen-2}
\left\{
\begin{array}{rl}
\partial_t z+(w^S\cdot \nabla)z+(z\cdot \nabla)w^S-\nu \Delta z+\nabla q=\mathds{1}_{\mathcal{O}} v +f  & \text{in}\ (0,\infty)\times \Omega,\\
\nabla \cdot z=0& \text{in} \ (0,\infty)\times \Omega,\\
z=0 & \text{on} \ (0,\infty)\times \partial\Omega,\\
z(0,\cdot)=z^0 & \text{in} \ \Omega.
\end{array}
\right.
\end{equation}
\begin{Theorem}\label{staboseen-2}
Assume $\sigma>0$ and let us consider $v$ given by \eqref{fb1-oseen}.
Then for any $z^0 \in \mathbb{V}$ and for any $f\in L^2_\sigma(0,\infty;\mathbb{H})$
the solution of \eqref{oseen-2} satisfies 
$$
z\in L^2_\sigma (0,\infty;[H^2(\Omega)]^3)\cap C^0_\sigma ([0,\infty);[H^1(\Omega)]^3)\cap H^1_\sigma (0,\infty;[L^2(\Omega)]^3),
$$
and 
\begin{equation}\label{ijt210}
\|z\|_{L^2_\sigma (0,\infty;[H^2(\Omega)]^3)\cap C^0_\sigma ([0,\infty);[H^1(\Omega)]^3)\cap H^1_\sigma (0,\infty;[L^2(\Omega)]^3)}
\leq C\left(\|z^0\|_{[H^1(\Omega)]^3} + \|f\|_{L^2_{\sigma}(0,\infty;[L^2(\Omega)]^3)}\right).
\end{equation}
\end{Theorem}

\section{Local feedback distributed stabilization of the Navier-Stokes system}\label{sec_navier}
We use the same notation as in the previous section. We consider the stabilization of the Navier-Stokes system with internal control:
\begin{equation}
\label{navier}
\left\{
\begin{array}{rl}
\partial_t \widetilde z+(\widetilde z \cdot \nabla)\widetilde z-\nu \Delta \widetilde z+\nabla \widetilde q=\mathds{1}_{\mathcal{O}} v+f^S  & \text{in}\ (0,\infty)\times \Omega,\\
\nabla \cdot \widetilde z=0& \text{in} \ (0,\infty)\times \Omega,\\
\widetilde z=b^S & \text{on} \ (0,\infty)\times \partial\Omega,\\
\widetilde z(0,\cdot)=\widetilde z^0 & \text{in} \ \Omega,
\end{array}
\right.
\end{equation}
around the stationary state
\begin{equation}
\label{navier-S}
\left\{
\begin{array}{rl}
(w^S \cdot \nabla) w^S-\nu \Delta w^S+\nabla r^S=f^S  & \text{in}\ \Omega,\\
\nabla \cdot w^S=0& \text{in} \ \Omega,\\
w^S=b^S & \text{on} \ \partial\Omega.
\end{array}
\right.
\end{equation}
We assume that $(w^S,r^S)$ is a solution of \eqref{navier-S} such that $w^S\in [H^2(\Omega)]^3$ as in the previous section.
The functions $f^S\in [L^2(\Omega)]^3$ and $b^S\in [W^{3/2}(\partial \Omega)]^3$ are independent of time.

We define 
$$
z=\widetilde z-w^S, \quad q=\widetilde q-r^S, \quad z^0=\widetilde z^0-w^S,
$$
so that 
\begin{equation}
\label{navier-diff}
\left\{
\begin{array}{rl}
\partial_t z+(w^S\cdot \nabla)z+(z\cdot \nabla)w^S-\nu \Delta z+\nabla q=\mathds{1}_{\mathcal{O}} v -(z\cdot\nabla) z & \text{in}\ (0,\infty)\times \Omega,\\
\nabla \cdot z=0& \text{in} \ (0,\infty)\times \Omega,\\
z=0 & \text{on} \ (0,\infty)\times \partial\Omega,\\
z(0,\cdot)=z^0 & \text{in} \ \Omega.
\end{array}
\right.
\end{equation}

Then we consider the following mapping
$$
\mathcal{Z} : L^2_\sigma(0,\infty;[L^2(\Omega)]^3) \to L^2_\sigma(0,\infty;[L^2(\Omega)]^3),
\quad
f\mapsto -(z\cdot\nabla) z,
$$
where $z$ is the solution given in \cref{staboseen-2}, associated with $z^0\in \mathbb{V}$ and $f\in L^2_\sigma(0,\infty;\mathbb{H})$.
Then by standard Sobolev embeddings, we find
\begin{equation}\label{nonlin}
\| z^1\cdot\nabla z^2\|_{L^2_\sigma(0,\infty;[L^2(\Omega)]^3)}
\\
\leq C \|z^1\|_{C^0_\sigma ([0,\infty);[H^1(\Omega)]^3)}\|z^2\|_{L^2_\sigma (0,\infty;[H^2(\Omega)]^3)}.
\end{equation}
Thus $\mathcal{Z}$ is well-defined. Let us set
$$
R=\|z^0\|_{[H^1(\Omega)]^3},
$$
and
$$
B_R = \left\{f\in L^2_\sigma(0,\infty;[L^2(\Omega)]^3) \ ; \ \|f\|_{L^2_\sigma(0,\infty;[L^2(\Omega)]^3)} \leq R \right\}.
$$
Then from \eqref{nonlin} and \eqref{ijt210},
$$
\|\mathcal{Z}(f)\|_{L^2_\sigma(0,\infty;[L^2(\Omega)]^3)} \leq 4CR^2,
$$
and $B_R$ is invariant by $\mathcal{Z}$ for $R$ small enough. Similarly, using \eqref{nonlin} and \eqref{ijt210},
for any $f^1,f^2\in B_R$, then
$$
\|\mathcal{Z}(f^1)-\mathcal{Z}(f^2)\|_{L^2_\sigma(0,\infty;[L^2(\Omega)]^3)} \leq 2CR \|f^1-f^2\|_{L^2_\sigma(0,\infty;[L^2(\Omega)]^3)},
$$
and thus $\mathcal{Z}$ is a strict contraction on $B_R$ for $R$ small enough. We thus deduce that $\mathcal{Z}$ admits a fixed point $f$ for
$\|z^0\|_{[H^1(\Omega)]^3}$ small enough and we notice that the solution 
$z$ given in \cref{staboseen-2}, associated with $z^0\in \mathbb{V}$ and $f\in L^2_\sigma(0,\infty;[L^2(\Omega)]^3)$ is a solution 
of \eqref{navier-diff}.

We have obtained the following local stabilization result for the Navier-Stokes system with internal control with delay:
\begin{Theorem}\label{thmNS}
Assume $\sigma>0$ and let us define $N_+$ by \eqref{Nc}.
Then there exist $K\in L^\infty_{\rm loc}(\mathbb{R}^2;\mathcal{L}(\mathbb{H}))$, 
$\zeta_k\in \mathcal{D}(A^*)$, $v_k\in [L^2(\mathcal{O})]^3$, $k=1,\ldots,N_+$ and $R>0$, such that
for any 
$$
\widetilde{z}^0\in [H^1(\Omega)]^3, \quad \nabla \cdot \widetilde{z}^0=0 \quad \text{in} \ \Omega,
\quad 
\widetilde z^0=b^S \  \text{on} \ \partial\Omega,
$$
and
$$
\|\widetilde{z}^0-w^S \|_{[H^1(\Omega)]^3}\leq R,
$$
there exists a unique solution $z$ of \eqref{navier} with
\begin{equation}
v(t)=\mathds{1}_{[\tau,+\infty)}(t) \sum_{k=1}^{N_+} \left(\int_{\Omega} \left[(\widetilde z-w^S)(t-\tau)+\int_0^{t-\tau} K(t-\tau,s)(\widetilde z-w^S)(s) \ ds\right]\zeta_k \ dx\right) v_k,
\end{equation}
satisfying
$$
\widetilde z-w^S\in L^2_\sigma (0,\infty;[H^2(\Omega)]^3)\cap C^0_\sigma ([0,\infty);[H^1(\Omega)]^3)\cap H^1_\sigma (0,\infty;[L^2(\Omega)]^3).
$$
Moreover we have the estimate
\begin{equation}\label{ijt213}
\|\widetilde z-w^S\|_{L^2_\sigma (0,\infty;[H^2(\Omega)]^3)\cap C^0_\sigma ([0,\infty);[H^1(\Omega)]^3)\cap H^1_\sigma (0,\infty;[L^2(\Omega)]^3)}
\leq C\|\widetilde z^0-w^S\|_{[H^1(\Omega)]^3}.
\end{equation}
\end{Theorem}

\section*{Acknowledgments.} 
The two first authors were partially supported by the ANR research project IFSMACS (ANR-15-CE40-0010). 
The third author was partially supported by the ANR research projects ISDEEC (ANR-16-CE40-0013) and ANR ODISSE (ANR-19-CE48-0004-01).

\bibliographystyle{plain}
\bibliography{biblio}
	
\end{document}